\documentclass[11pt]{article}
\usepackage{geometry}                
\usepackage{graphicx}
\usepackage{amssymb}
\usepackage{epstopdf}
\usepackage{amsmath,amsthm,amscd,amssymb}
\usepackage{latexsym}
\numberwithin{equation}{section}

\theoremstyle{plain}
\newtheorem{theorem}{Theorem}[section]
\newtheorem{lemma}[theorem]{Lemma}
\newtheorem{corollary}[theorem]{Corollary}
\newtheorem{proposition}[theorem]{Proposition}
\newtheorem{conjecture}[theorem]{Conjecture}

\theoremstyle{definition}
\newtheorem{definition}[theorem]{Definition}

\newtheorem{example}[theorem]{Example}

\theoremstyle{remark}

\newtheorem{case[theorem]}{Case}

\def\bb #1{ {\mathbb #1} }

\title{Tiling, circle packing and exponential sums over finite fields }
\author{C. D. Haessig, A. Iosevich, J. Pakianathan, S. Robins and L. Vaicunas}

\begin{document}
\maketitle

\begin{abstract} We study the problem of tiling and packing in vector spaces over finite fields, its connections with zeroes of classical exponential sums, and with the Jacobian conjecture.  \end{abstract} 

\tableofcontents

\section{Introduction}
\subsection{Tilings}

Tiling is one of the most diverse and ubiquitous concepts of modern mathematics. In Euclidean space, we say that a domain $\Omega \subset {\Bbb R}^d$ tiles by $A \subset {\Bbb R}^d$ if 
$$ \sum_{a \in A} \Omega(x-a)=1 \ \text{for a.e.} \ x \in {\Bbb R}^d,$$ where $\Omega(x)$ denotes the characteristic function of 
$\Omega$. Much work has been done over the years on this subject matter in a variety of contexts. See for example, \cite{A05}, \cite{ER61}, \cite{M80} and the references contained therein. 

The purpose of this paper is to study tiling in vector spaces over finite fields. Let $E \subset {\Bbb F}_q^d$, the $d$-dimensional vector space over the finite field with $q$ elements. 

\begin{definition} \label{ktilingdef} Let $A \subset {\Bbb F}_q^d$. We say that $E$ $k$-tiles ${\Bbb F}_q^d$ by $A$ if 
$$ \sum_{a \in A} E(x-a)=k \ \text{for every} \ x \in {\Bbb F}_q^d.$$ 
\end{definition} 

Note that $k$-tiling means that each point in ${\Bbb F}_q^d$ is covered exactly $k$ times by some translate of the set $E$.  There is an immediate duality here that follows from the definition: $E$ $k$-tiles ${\Bbb F}_q^d$ by $A$ if and only if  $A$ $k$-tiles ${\Bbb F}_q^d$ by $E$.  For some of the historical perspectives and the growing body of work on multiple tilings, we refer the reader to \cite{G00}, \cite{GKRS13}, \cite{GRS12}, and \cite{R79}.

We first provide a simple characterization of $k$-tiling in terms of the coefficients of the Fourier transform. Given $U \subset {\Bbb F}_q^d$, $|U|$ denotes the number of elements in $U$. 

\begin{theorem} (Fourier Characterization) \label{characterization} 
The set $E \subset {\Bbb F}_q^d$ multi-tiles ${\Bbb F}_q^d$ by $A \subset {\Bbb F}_q^d$ at level $k$ if and only if 
$$ |A| \cdot |E|=kq^d$$ and 
$$ \widehat{E}(m) \cdot \widehat{A}(m)=0 \ \text{for every} \ m \not=\vec{0}.$$ 

\bigskip \noindent
Here and throughout, given $f: {\Bbb F}_q^d \to {\Bbb C}$, we recall the discrete Fourier transform
$$ \widehat{f}(m) \equiv q^{-d} \sum_{x \in {\Bbb F}_q^d} \chi(-x \cdot m) f(x),$$ where $\chi$ is a principal additive character on 
${\Bbb F}_q$. 
\end{theorem} 

\vskip.125in 

\noindent
In order to illustrate Theorem \ref{characterization}, let's consider the following examples. 

\begin{example} \label{graph} Suppose that $E=\{(t,f(t)): t \in {\Bbb F}_q \}$, where $f: {\Bbb F}_q \to {\Bbb F}_q$ is arbitrary. Then it is clear that $E$ tiles by $A=\{(0,c): c \in {\Bbb F}_q\}$. To see how this fits in with Theorem \ref{characterization}, observe that 
$$ \widehat{A}(m)=q^{-2} \sum_{c \in {\Bbb F}_q} \chi(-cm_2).$$ 

It follows directly from the orthogonality relations for characters that this quantity is $0$ if $m_2 \not=0$  and $q^{-1}$ if $m_2=0$. In other words, $\widehat{A}(m)$ is non-zero on 
$\{(t,0): t \in {\Bbb F}_q\}$ and $0$ everywhere else. On the other hand, 
$$ \widehat{E}(m_1,0)=q^{-2} \sum_{x_1 \in {\Bbb F}_q} \chi(-x_1m_1)=0 \ \text{if} \ m_1 \not=0.$$ 

Therefore  $\widehat{E}(m) \cdot \widehat{A}(m)=0$ if $m \not=\vec{0}$. 

\vskip.125in 

This example easily generalizes to sets in ${\Bbb F}_q^d$ of the form 
$$ E=\{(x_1, x_2, \dots, x_j, f_1(\vec{x}), \dots, f_{d-j}(\vec{x}) \}, $$ where $\vec{x}=(x_1, \dots, x_j)$. 

\end{example} 

\begin{example} \label{subspace} Let $H_j$ denote a $j$-dimensional affine subspace of ${\Bbb F}_q^d$, $1 \leq j \leq d-1$. By elementary linear algebra, $H_j$ tiles by $H_j^{\perp}$ at level $1$. Let us see how this fits in with Theorem \ref{characterization}. We have 
\begin{equation} \label{subspaceft} \widehat{H}_j(m)=q^{-(d-j)} H_j^{\perp}(m), \end{equation} where $H_j^{\perp}$ is the orthogonal subspace. It follows that 
$$ \widehat{H}_j(m) \widehat{H_j}^{\perp}(m)=q^{-(d-j)} \cdot q^{-j} \cdot H_j^{\perp}(m) \cdot H_j(m)=0 \ \text{if} \ m \not=\vec{0}.$$ 
\end{example} 

\vskip.125in 

\begin{example} \label{multitileex} Let $p$ be a prime and let $E=\{0,1,2, \dots, k-1\}$, $k \ge 2$. Let $A={\Bbb Z}_p$. Then $E$ $k$-tiles 
${\Bbb Z}_p$ by $A$ and does not $j$-tile ${\Bbb Z}_p$ for any $j<k$. It is instructive to note here that since $A={\Bbb Z}_p$ and $A$ is viewed as a subset of ${\Bbb Z}_p$, $\widehat{A}$ vanishes identically away from $0$. \end{example} 

\vskip.125in 

The examples above give us some hints about characterization of $1$-tiings, which we now proceed to classify. 

\begin{definition} \label{graphtilingdef} Suppose that $E$ $1$-tiles ${\Bbb F}_q^d$ by $A$. We say that $E \subset {\Bbb F}_q^d$ is a {\bf graph} if it there is a coordinate decomposition 
$$ {\Bbb F}_q^d = {\mathbb F}_q^s \oplus \mathbb{F}^{d-s}_q =\{ (x_1, \dots, x_s, y_1, \dots, y_{d-s}) \}$$ such that 
$$E = \{ (\hat{x}, f(\hat{x}): \hat{x} \in {\Bbb F}_q^s \} = Graph(f)$$ for some function $f: \mathbb{F}_q^s \to \mathbb{F}_q^{d-s}$ with respect to these coordinates. 

Note given such an $E=Graph(f)$, we can come up with sets $E,A$ such that $E$ $1$-tiles ${\Bbb F}_q^d$ by $A$. Here 
$$A=\vec{0} \oplus \mathbb{F}_q^{d-s} = \left\{ (\vec{0},\hat{y}): \hat{y} \in {\Bbb F}_q^{d-s} \right\},$$ so all such graphs give $1$-tilings. \end{definition}

As some of our results only apply for prime fields $\mathbb{F}_p$ where $p$ is a prime, we should point out that any tiling pair $(E,A)$ 
in $\mathbb{F}_q^d$ with $q=p^{\ell}$ can also be viewed as a tiling pair in $\mathbb{F}_p^{d\ell}$ as $\mathbb{F}_q \cong \mathbb{F}_p^\ell$ as additive groups. 
Thus as tiling only uses the additive structure of the ambient space, it is more informative to work only in vector spaces $V$ over the prime fields $\mathbb{F}_p$, as the dimension of $V$ over $\mathbb{F}_p$ more accurately characterizes structural questions regarding tilings than the dimension of $V$ over other finite fields. 

The following is a characterization of $1$-tilings in ${\Bbb F}_p^2$ when $p$ is a prime. 

\begin{theorem}[Classification of $1$-tilings of the plane] \label{1tilecharacterization} Let $q=p$, a prime and suppose that $E$ $1$-tiles ${\Bbb F}_p^2$ by $A$. Then $E$ is a graph. Thus either $|E|=1$ and $A={\Bbb F}_p^2$, $E={\Bbb F}_p^2$ or $|A|=1$, or there is a function $f: {\mathbb F}_p \to {\Bbb F}_p$ such that $E=Graph(f)$ with respect to some choice of coordinate axis decomposition for ${\mathbb F}_p^2$. Note the function $f$ can always be given by a polynomial of degree at most $p-1$.
\end{theorem}

\vskip.125in 

A similar argument yields a characterization of $k$-tilings for ${\Bbb F}_p^2$ when $p$ is a prime. 
\begin{theorem}[Classification of $k$-tilings of the plane $\mathbb{F}_p^2$ ] \label{ktilecharacterization} Let $q=p$, a prime and suppose that $E$ $k$-tiles ${\Bbb F}_p^2$.Then either $|E|=k$ and $A={\Bbb F}_p^2$, $E={\Bbb F}_p^2$ and $|A|=k$, or 
$|E|=sp$ for some $1 \leq s \leq p-1$ such that $s$ divides $k$ and $E$ is the union of $s$ disjoint graphs 
$Graph(f_1), \dots, Graph(f_s)$.
\end{theorem}

Theorem~\ref{1tilecharacterization} gives much more structural information  about tiling sets than the simple divisibility condition on their order, i.e.,  
$|E|=1, p$ or $p^2$ in the case when $|E|=p$. This is because most sets of size $p$ are not graphs as the next example explains.

\begin{example} We give an example of $p$ points in $\mathbb{F}_p^2$ which cannot be expressed as $Graph(f)$ for any $f$ or equivalently does not equidistribute on any collection of $p$ parallel lines. Note points equidistribute on $p$ parallel lines if and only if the difference of any two does not lie in a parallel line through the origin. Thus to give examples of $p$ points that do not equidistribute in any direction, it is equivalent to give a collection of points whose pairwise differences generate all directions. Note there are $p+1$ directions in $\mathbb{F}_p^2$ which can be encoded by the slopes of the corresponding lines $\{0,1,\dots, p-1, \infty \}$.

For $p=5$ a collection of $5$ points in $\mathbb{F}_5^2$ which do not equidistribute in any direction is given by

$$ \{ (0,0), (1,1), (2,3), (3,1), (2,4) \} $$ since their pairwise differences generate every possible slope (direction) as the reader can easily check.

Note in general, the number of directions in $\mathbb{F}_p^2$ is $p+1$ while a set of $p$ points has $\binom{p}{2}=\Theta(p^2)$ pairs of distinct points 
so one expects this last example to be ``generic'' amongst subsets of size $p$ in $\mathbb{F}_p^2$, i.e., that most subsets of size $p$ determine all directions and hence are not graphs. Indeed there are $\binom{p^2}{p}=\frac{p^{2}(p^2-1)\dots(p^2-p+1)}{p!}$ sets of size $p$ in $\mathbb{F}_p^2$ while one has at most $(p+1)p^p$ of these  are graphs (choose a direction and then one point per hyperplane in the family of hyperplanes perpendicular to that direction). One can compute
$$
\frac{\binom{p^2}{p}}{(p+1)p^p}=\frac{p(p-\frac{1}{p})(p-\frac{2}{p}) \dots (p-\frac{p-1}{p})}{(p+1)!} \geq \frac{(p-1)^p}{(p+1)!} \to \infty
$$
as $p \to \infty$.
\end{example}

\vskip.125in 

The proof of Theorem~\ref{1tilecharacterization} carries over to $3$-dimensions as long as we weaken the conclusion slightly by showing 
that either $E$ or $A$ is a graph and hence the tiling is graphical (though $E$ might not be the graph).

\begin{theorem}[Classification of $1$-tilings in $3$-space]\label{1tilecharacterization3space} Let $q=p$ a prime and suppose 
$E$ $1$-tiles ${\Bbb F}_p^3$ by $A$. Then $(E,A)$ is a graphical tiling i.e., either $E$ or $A$ is a graph. More generally, if $(E,A)$ is a tiling 
pair in $\mathbb{F}_p^d$ with $|E|=p$ or $p^{d-1}$ then the tiling is graphical.
\end{theorem}

In higher dimensions, the results of Theorems~\ref{1tilecharacterization} and~\ref{1tilecharacterization3space} seemingly fail i.e., not every tiling pair 
is graphical at least over cyclic groups of nonprime order. For example in \cite{KM06}, a tiling of $\mathbb{Z}_6^5 \times \mathbb{Z}_{15}$ which is not spectral is constructed. Such a tiling cannot be graphical as all graphs are 
spectral sets. However it is not clear exactly in which dimension the graphical structure of tilings breaks down if we consider only vector spaces over prime fields $\mathbb{Z}_p$. The first possibility not covered by our methods of a non-graphical tiling would be $(E,A) \subseteq \mathbb{F}_p^4$ 
with $|E|=|A|=p^2$ but we have not found any such explicit example.

\subsection{Packings} When tiling or multi-tiling is not possible, we can still ask for the largest proportion of ${\Bbb F}_q^d$ that various disjoint translates of $E \subset {\Bbb F}_q^d$ can cover. 

\begin{definition} \label{packingdef} We say that $E$ packs ${\Bbb F}_q^d$ by $A$ if $A \subset {\Bbb F}_q^d$ with the property that 
$$(E+a) \cap (E+a')=\emptyset \ \text{for all} \ a \not=a', \ a,a' \in A.$$ 
\end{definition} 

\begin{definition} \label{optpackingdef} We say that $A$ is an optimal packing set for $E$ if $A$ is a set of largest possible size such that $E$ packs ${\Bbb F}_q^d$ by $A$. \end{definition} 

\begin{definition} \label{densitypackingdef} The density of packing of $E \subset {\Bbb F}_q^d$ by $A \subset {\Bbb F}_q^d$ is the ratio 
$\frac{|E||A|}{q^d}$. \end{definition} 

We start out with a very negative result. 
\begin{theorem} \label{4dsphere} Let $E=S_t=\{x \in {\Bbb F}_q^d: ||x||=t \}$, where $||x||=x_1^2+\dots+x_d^2$. \begin{itemize} 
\item If $d \ge 4$ and $t \not=0$, then the optimal packing of $S_t$ has size $1$. 

\end{itemize} 

\end{theorem} 

We now specialize to the two-dimensional case. 

\begin{definition} A circle of ``radius'' $R$ and center $(u,v) \in \mathbb{F}_q^2$ is the affine variety given by
$$ C_R((u,v))=\left\{ (x,y) \in \mathbb{F}_q^2: (x-u)^2 + (y-v)^2 = R \right\}$$
$$=\{ (x,y) \in {\Bbb F}_q^2: ||(x,y)-(u,v)||=R \}. $$

Note that this notion of radius is the square of the usual one but is more suitable for work in general fields. The packing number $P(q,c)$ is defined as the maximum number of pairwise disjoint circles of radius $c$ that you can fit into the plane $\mathbb{F}_q^2$. Scaling by $M \neq 0$ shows that $P(q,M^2c)=P(q,c)$. We will often abbreviate $P(q,1)=P(q)$ when referring to the packing number of unit circles.
\end{definition}

In \cite{BIP14}, the following lemma about completing line segments into triangles in the plane was proven. It will be very useful in our study of circle packings:

\begin{lemma}\label{lem: triangle}
Fix a field $\mathbb{F}$ of characteristic not equal to two and let $P$ be the plane $\mathbb{F}^2$. 
Let $(x_1,x_2)$ be a line segment of length $||x_1-x_2||=(x_1-x_2) \cdot (x_1-x_2) =\ell_1 \neq 0$ in the plane $P$. This segment 
can be extended into exactly $\mu$ triangles $(x_1,x_2,x_3)$ with 
$||x_2-x_3||=\ell_2$ and $||x_3-x_1||=\ell_3$ given where 
$$
\mu=\begin{cases}
2 \text{ if } 4\sigma_2-\sigma_1^2 \text{ is a nonzero square in } \mathbb{F} \\
1 \text{ if } 4\sigma_2-\sigma_1^2\text{ is zero} \\
0 \text{ if } 4\sigma_2-\sigma_1^2 \text{ is a nonsquare in } \mathbb{F}
\end{cases}
$$
and $\sigma_1=\ell_1+\ell_2+\ell_3$, $\sigma_2=\ell_1\ell_2+\ell_2\ell_3+\ell_3\ell_1$.
\end{lemma}

This lemma immediately leads to a criterion for when two circles intersect in the plane:

\begin{corollary}
Let $C_1$ be a circle of radius $\ell_1$ and $C_2$ be a circle of radius $\ell_2$ and let $\ell_3 = ||x_1-x_2|| \neq 0$ be the distance between their centers 
$x_1$ and $x_2$. Then $C_1$ and $C_2$ intersect in $\mu$ points where 
$$
\mu=\begin{cases}
2 \text{ if } 4\sigma_2-\sigma_1^2 \text{ is a nonzero square in } \mathbb{F} \\
1 \text{ if } 4\sigma_2-\sigma_1^2\text{ is zero} \\
0 \text{ if } 4\sigma_2-\sigma_1^2 \text{ is a nonsquare in } \mathbb{F}
\end{cases}
$$
and $\sigma_1=\ell_1+\ell_2+\ell_3$, $\sigma_2=\ell_1\ell_2+\ell_2\ell_3+\ell_3\ell_1$. When $\ell_1=\ell_2=c$ and $\ell_3=R \neq 0$ we find two 
circles of radius $c$ and distance $R$ between their centers intersect in $\mu$ points where 
$$
\mu=\begin{cases}
2 \text{ if } R(4c-R) \text{ is a nonzero square in } \mathbb{F} \\
1 \text{ if } R(4c-R) \text{ is zero i.e., } $R=4c$ \\
0 \text{ if } R(4c-R) \text{ is a nonsquare in } \mathbb{F}
\end{cases}
$$
\end{corollary}

The proof is very simple, so we give it here. The first part is clear and the second follows from the following computation. Setting $\ell_1=\ell_2=c$ and $\ell_3=R \neq 0$ yields $\sigma_1 = 2c + R$ and $\sigma_2 = 2cR+c^2$. Thus 
$$ 4\sigma_2-\sigma_1^2 = 4(2cR+c^2)-(2c+R)^2=4cR-R^2=R(4c-R).$$

\begin{example}[packing of circles in $\mathbb{F}_3^2$]
The function $f: \mathbb{F}_3 \to \mathbb{F}_3$ given by $f(R)=R(4-R)=R(1-R)$ takes only the values $\{0,1\}$ and so is always a square.
Thus any two unit circles in the plane $P=\mathbb{F}_3^2$ intersect. Thus the unit packing number $P(3)=P(3,1)$ is equal to $1$ as you can pack at most 
$1$ unit circle into the plane $\mathbb{F}_3^2$. The function $g: \mathbb{F}_3 \to \mathbb{F}_3$ given by $g(R)=R(4(2)-R)=R(-1-R)$ takes only the values 
$\{0,1\}$ also so $P(3,2)=1$ also.
\end{example}

Using the same results, we get a criterion for the existence of three disjoint circles of radius $c$ in the plane:

\begin{proposition} There is a packing of three disjoint circles of radius $c$ in $\mathbb{F}_q^2$ whose centers make a triangle with nonzero sidelengths $\ell_1, \ell_2, \ell_3$  if and only if 
$$ \ell_1(4c-\ell_1), \ell_2(4c-\ell_2), \ell_3(4c-\ell_3) $$
are non squares in $\mathbb{F}_q$ and $4\sigma_2 - \sigma_1^2$ is a square where $\sigma_2=\ell_1\ell_2+\ell_2\ell_3+\ell_3\ell_1, \sigma_1=\ell_1+\ell_2+\ell_3$.
\end{proposition}

Once again, the proof is straightforward. It follows from the previously mentioned results since such a packing exists in ${\Bbb F}_q^2$ if and only if the triangle made by the centers exists in ${\Bbb F}_q^2$ and the circles of radius $c$ centered at the vertices of the triangle are pairwise disjoint. 

This generalizes immediately using the same argument, to a brute force algorithm to determine whether one can pack $k$ disjoint circles of radius $c$ in a finite plane:

\begin{theorem}[Primitive Circle Packing Algorithm]
Fix a finite field $\mathbb{F}_q$ of odd characteristic and nonzero $c \in \mathbb{F}_q$. Let $f: \mathbb{F}_q \to \mathbb{F}_q$ be given 
by $f(R)=R(4c-R)$ and let $S=\{ R \in \mathbb{F}_q | f(R) \text{ is not a square in } \mathbb{F}_q\}$. Note $0 \notin S$. Then it is possible to pack $k$ distinct circles of radius $c$ in the plane $P=\mathbb{F}_q^2$ (with nonzero distance between centers) if and only if it is possible to find $k$ distinct vertices in $P$ that form a $(k-1)$-simplex whose edge lengths all lie in $S$. In the case this holds, the packing is achieved by placing the $k$-circles at the vertices of the $(k-1)$-simplex.
\end{theorem}

Note when $q=1 \mod 4$, it is possible to have zero distance between centers of distinct circles and this must be excluded in the theorem. The theorem is a complete characterization of packing only when $q=3 \mod 4$.

When searching for such a $(k-1)$ simplex as mentioned in the algorithm, one can assume the first vertex is the origin and hence consider only the 
$O((q^2)^{k-1})$ possible $(k-1)$ simplices pinned at the origin. For each of these, a calculation of $\binom{k}{2}$ distances will determine the edge lengths 
of this simplex and determine if it will work or not. Thus a brute force search will essentially use at most $O(k^2 q^{2k-2})$ basic steps to determine if 
a $k$-packing of circles of radius $c$ exists in the plane $P=\mathbb{F}_q^2$ or not, and if it does exist, will produce such a packing.
This is a polynomial time algorithm for any fixed $k$. However as the only a priori upper bound on $P(q,c)$ is $\frac{q^2}{q-1}=O(q)$ as the size of a circle is at least 
$q-1$, using this algorithm to determine $P(q,c)$ is not feasible as it would have to check $1 \leq k \leq q$ cases leading to $O(q^2q^{2q-2})=O(q^{2q})$ efficiency 
which is not feasible.

When the distance between the distinct circles in the packing is allowed to be zero, in fact larger packings can be achieved. This can be done only when 
$q = 1 \mod 4$. In this case there is a primitive 4th root of unity $i$ in $\mathbb{F}_q$ and the two lines $\{ (t,it) | t \in \mathbb{F}_q \}$ and 
$\{ (t,-it) | t \in \mathbb{F}_q \}$ are isotropic ($||v||=0$ for $v$ on these lines) and together form the circle of radius zero about the origin.
It turns out for any nonzero $c$, the $q$ circles of radius $c$ centered on the points of an isotropic line are all disjoint and hence form a 
``degenerate'' isotropic packing of nearly optimal size.

\begin{theorem}[Isotropic circle packing] \label{icp} 
Let $q = 1 \mod 4$ and let $P=\mathbb{F}_q^2$ be the affine plane over $\mathbb{F}_q$. 
If $i$ is a primitive $4$th root of unity in $\mathbb{F}_q$ then the lines $\{(t,it) | t \in \mathbb{F}_q \}$ and $\{ (t,-it) | t \in \mathbb{F}_q \}$ are the two isotropic lines in this plane. For any $c \neq 0$, we have 
$$ q \leq P(q,c) \leq q+1. $$

A packing of $q$ circles of radius $c$ can be achieved by taking the $q$ circles of radius $c$ centered at the $q$ points on an isotropic line. The complement of  this packing is nothing other than the isotropic line itself.
\end{theorem}

Note the isotropic packing of the last theorem is maximal as one can show that the complementary isotropic line cannot contain any circle of nonzero radius. However this does not necessary preclude the existence of a completely different maximum circle packing achieving $P(q,c)=q+1$.

\vskip.125in

\subsection{Algebraic Tilings}

In this section we study algebraic $1$-tilings of affine space over algebraically closed fields. We also relate some of these results to the well-known Jacobian conjectures (see e.g. \cite{FMV14}).


\begin{definition} Let $k$ be a field and $\bb A^d$ denote $d$-dimensional affine space over $k$. By a regular automorphism of $\bb A^d$ we mean a 
bijective function $F: \bb A^d \to \bb A^d$ such that both $F$ and $F^{-1}$ have polynomial coordinate functions. 
\end{definition}

The study of regular automorphisms of affine space has a long history, and a useful criterion in determining whether a polynomial function 
$F: \bb A^d \to \bb A^d$ has a polynomial inverse in given conjecturally (over algebraically closed fields) via the famous Jacobian conjecture. 

To state this conjecture, notice if $F^{-1}$ exists and is polynomial then both Jacobian matrices $DF$ and $DF^{-1}$ have polynomial entries, and taking determinants we see that $JF \cdot JF^{-1} = 1$ where $JF$ and $JF^{-1}$ are polynomials. Hence $JF$ has to be a nonzero constant as the only units in a polynomial ring are nonzero constants. The Jacobian conjecture is that this necessary condition is also sufficient:

\begin{conjecture}[Jacobian Conjecture]
Let $k$ be an algebraically closed field and $\bb A^d$ be $d$-dimensional affine space over $k$. A polynomial function $F: \bb A^d \to \bb A^d$ has a polynomial inverse if and only if the Jacobian determinant $JF$ is a nonzero constant (that is, $JF \in k^*$).
\end{conjecture}

The Jacobian conjecture reduces to the case over $\mathbb{C}$ in characteristic zero but is still open in all dimensions $d \geq 2$.

\begin{definition} Fix $k$ a field. Let $X$ and $Y$ be algebraic sets in $d$-dimensional affine space $\mathbb{A}^d_k$. Then $(X,Y)$ is an \emph{algebraic $1$-tiling} if 
the map $\theta: X \times Y \to \mathbb{A}^d_k$ given by $\theta(x,y)=x+y$ is an isomorphism of algebraic sets.
\end{definition}

\begin{theorem} \label{abyth} 
Let $k$ be an algebraically closed field, and suppose $(X, Y)$ algebraically 1-tiles $\bb A^d_{k}$. Then dim $X$ + dim $Y$ = $d$.  
\end{theorem}

\begin{proof}
By definition, the map 
\[
\theta: X \times_{k} Y \rightarrow \bb A^d_{k} \quad \text{via} \quad (x, y) \mapsto x+y
\]
is an isomorphism. Denote by $k[X]$ and $k[Y]$ the coordinate rings of $X$ and $Y$, respectively. By Noether's normalization lemma, there exist algebraically independent elements $x_1, \ldots, x_i \in k[X]$ and $y_1, \ldots, y_j \in k[Y]$ such that $k[X]$ and $k[Y]$ are finitely generated modules over the polynomial rings $k[x_1,\dots, x_i]$ and $k[y_1,\dots,y_j],$ respectively. Thus $k[X] \otimes k[Y]=k[\bb A^d]$ is a finitely generated 
module over $k[x_1,\dots,x_i] \otimes k[y_1,\dots,y_j] = k[x_1,\dots, x_i, y_1,\dots, y_j]$. Comparing Krull dimensions we find $i+j=d$ as claimed. 
\end{proof}

In the graphical classification of tilings over finite fields (e.g. Theorem \ref{1tilecharacterization}), either $X$ or $Y$ is isomorphic to a linear subspace of $\bb A^2$. While this may not occur for general $\bb A^d$, we expect considerable restrictions on what type of $X$ and $Y$ may occur since the algebraic sets $X$ and $Y$ must behave well under the ambient vector addition in $\bb A^d$. For instance, it may be the case that the coordinate rings $k[X]$ and $k[Y]$ are isomorphic to polynomial rings, say $k[X] \cong k[f_1,\dots, f_i]$ and $k[Y] \cong k[g_1,\dots, g_j]$; that is, $X$ and $Y$ are isomorphic to affine spaces $\bb A^i$ and $\bb A^j$, respectively. If we express the generators $f_s$ and $g_t$ as polynomials in the coordinates $z_1, \dots z_d$ of $\bb A^d$, then the function 
\[
F(z_1, \dots, z_d) = (f_1(z_1,\dots, z_d), \dots, f_i(z_1,\dots, z_d), g_1(z_1, \ldots, z_d), \dots, g_j(z_1, \dots, z_d))
\] 
is a polynomial function $F: \bb A^d \to \bb A^d$ which is an isomorphism by construction. Thus it is a regular automorphism of $\bb A^d$. Note that $F$ takes $X$ in the original coordinate system to the set $\{ (f_1,\dots,f_i, 0, \dots, 0) \}$ and $Y$ to $\{ (0, \dots 0, g_1, \dots, g_j ) \}$ in the new coordinate system, which is analogous to the classification in Theorem \ref{1tilecharacterization}.

\vskip.125in 

\subsection{A polynomial approach to multi-tiling over finite fields} 
An alternate approach to the main problem may be taken by associating polynomials (in $d$ variables) to the sets $A$ and $E$. 
We begin by associating to each point $m \in {\Bbb F}_q^d$ the monomial $z^m:= z_1^{m_1} \cdots z_d^{m_d}$.   We notice that if we form the multivariate polynomial that encodes all of the information about the set $E$ by forming $\sum_{e \in E} z^e$, and then we translate the set $E$ by any element $a \in A$, this encoding corresponds to the multiplication  $z^a \sum_{e \in E} z^e$.  But  in order to respect the additive structure of the abelian group $\left( \Bbb Z/ q \Bbb Z\right)^d$ inside the exponents of these monomials, we will need to work in the ring  
$$\left( \Bbb Z/ q \Bbb Z\right)   [z_1, \dots, z_d] / \left( (z_1^q - 1), \dots, (z_d^q - 1) \right).$$

Note that when $q=p^{\ell}$, $p$ a prime, the finite field $\mathbb{F}_q$ has additive group isomorphic to $(\mathbb{Z}/p\mathbb{Z})^{\ell}$ so when 
applying results in this section to fields, $q=p$ a prime should be assumed.
 
For example, for $d=2, q=7$, we have $z_1^4z_2^5 = z_1^{11}z_2^5 \mod ((z_1^7 -1), (z_2^7 - 1))$.   We note that although we could work with $ \Bbb F_q [z_1, \dots, z_d]$ and its relevant quotients, we are here simply ignoring the multiplicative structure of the field and focusing on  its additive structure.
To simplify notation, we define the ideal 
$${\cal I}:= \left( (z_1^q - 1), \dots, (z_d^q - 1) \right)$$ in our ring 
$$\left( \Bbb Z/ q \Bbb Z\right) [z_1, \dots, z_d]$$ and work with multivariate polynomials mod $\cal I$.  We note that when we write 
$\frac{z_1^q -1}{z_1 - 1}$, we mean the polynomial  
$$1+ z_1 + z_1^2 + \cdots + z_2^{q-1},$$ so there are no rational functions here and hence we do not need to worry about singularities.   Another advantage of this polynomial approach is that we may differentiate as many times as we like with respect to each variable, obtaining many moment identities for each $d$ and $q$.

\begin{theorem}  \label{polysex} 
A set $E \subset {\Bbb F}_q^d$ multi-tiles  by a set of translation vectors $A \subset {\Bbb F}_q^d$, at level $k \ge 1$, if 
and only if 
\begin{equation} 
 \left(\sum_{e \in E  } z^e \right) \left(\sum_{a \in A} z^a \right)   =
  k \prod_{i=1}^d   \frac{z_i^{q} -1}{z_i - 1}   \ \mod  {\cal I}.   
\end{equation}
\end{theorem}

We note that the Fourier result follows immediately from Theorem \eqref{polysex} by specializing all of the coordinates of $z \in \Bbb C^d$ to be roots of unity 
$$z_j := e^{2\pi i k_j  / q},$$ and in this case all elements of the ideal $\cal I$ vanish identically, considering everything over $\Bbb C$ now. Furthermore, letting each $z_j = 1$ in Theorem \eqref{polysex}  retrieves the arithmetic constraint of Theorem \ref{characterization}, namely that  $|A| |E| = k q^d$.

With the same notation of tiling at level $k$, we have the following consequence of the polynomial identity given by Theorem \ref{polysex}, using differentiation. 

\begin{corollary} \label{sumpolysex} 
\begin{align}
 &  |A|  \left(\sum_{e \in E  } e  \right)
 + |E|  \left(\sum_{a \in A} a \right)  = 0 \text{ in } \left( \Bbb Z/{q \Bbb Z} \right)^d.
\end{align}
\end{corollary}

We may also obtain higher moment identities by further differentiation of the identity of Theorem \ref{polysex}, and here is the next such moment identity. 

\begin{corollary} \label{sumpolysex2} 
\begin{align}
 &    \left(\sum_{e \in E  } e_j^2  \right)|A| 
 +  \left(\sum_{a \in A} a_j^2  \right)  |E| 
 +  2 \left(\sum_{e \in E  } e_j \right)    \left(\sum_{a \in A} a_j  \right) \\
     &=  0   \mod q,
\end{align}
for each $1 \leq j \leq d$.
\end{corollary}

We note that higher moment formulas continue to exist, using similar methods, albeit they get more messy.   We've given here an analogue of the `mean' and of the `second moment' for any finite tiling sets $E$ and $A$.  

\vskip.125in 

\section{Proof of Theorem \ref{characterization}: Fourier characterization of tilings.} 

\begin{proof}   From the usual theory of the discrete Fourier transform, it is straightforward that 
\begin{equation} \label{inversion} f(x)=\sum_{m \in {\Bbb F}_q^d} \chi(x \cdot m) \widehat{f}(m) \end{equation} and 
\begin{equation} \label{plancherel} \sum_{m \in {\Bbb F}_q^d} {|\widehat{f}(m)|}^2=q^{-d} \sum_{x \in {\Bbb F}_q^d} {|f(x)|}^2. \end{equation} 

It follows that 
$$ k=\sum_{a \in A} E(x-a)=\sum_{a \in A} \sum_{m \in {\Bbb F}_q^d} \chi((x-a) \cdot m) \widehat{E}(m)$$
$$=q^d \sum_{m \in {\Bbb F}_q^d} \widehat{A}(m) \widehat{E}(m)$$
$$=|A||F|q^{-d}+\sum_{m \not=\vec{0}} \chi(x \cdot m) \widehat{A}(m) \widehat{E}(m).$$

It is not difficult to see that since $E$ multi-tiles by $A$, $|A||F|q^{-d}=k$, therefore 
$$ D(x)=\sum_{m \not=\vec{0}} \chi(x \cdot m) \widehat{A}(m) \widehat{E}(m) \equiv 0 \ \text{for all} \ x \in {\Bbb F}_q^d.$$ 

In particular, this means that $\sum_{x \in {\Bbb F}_q^d} {|D(x)|}^2=0$, so by (\ref{plancherel}), 
$$ \sum_{m \in {\Bbb F}_q^d} {|\widehat{D}(m)|}^2=0.$$ 

Observe that 
$$ \widehat{D}(\vec{0})=0,$$ therefore, 
$$ 0=\sum_{m \in {\Bbb F}_q^d} {|\widehat{D}(m)|}^2=\sum_{m \not=\vec{0}} {|\widehat{A}(m)|}^2 {|\widehat{E}(m)|}^2,$$ from which we instantly conclude that 
$$ \widehat{A}(m) \cdot \widehat{E}(m)=0 \ \text{for every} \ m \not=\vec{0}.$$ 

This completes the proof. 
\end{proof}

\section{Proof of Theorem \ref{1tilecharacterization}, Theorem \ref{ktilecharacterization} and Theorem~\ref{1tilecharacterization3space}: Tiling characterization in $\mathbb{F}_p^2$ and $\mathbb{F}_p^3$} 

In this section, $p$ is a prime.

\begin{lemma} \label{lemma: novanish}
Suppose that $E$ $k$-tiles ${\Bbb F}_p^d$ by $A$. Then $\hat{E}(m)$ is nonzero for all $m$ if and only if $|E|=k$ and $A={\Bbb F}_p^d$.
Furthermore $1 \leq k \leq p^d$. \end{lemma}

To see this, observe that since $\hat{E}(m) \hat{A}(m) = 0$ for all nonzero $m$ and $\hat{E}(m)$ is assumed nonzero, we conclude $\hat{A}(m)=0$ for all $m \not=\vec{0}$. Since $A(x)=\sum_{m \in {\Bbb F}_p^d} \chi(x \cdot m) \widehat{A}(m)$, we conclude that $A$ is a constant function and hence that $A$ must be the whole space since $A$ cannot be the empty set. This forces $E$ to consist of $k$ distinct points since $E$ $k$-tiles ${\Bbb F}_p^d$ by $A$. Since $E \subset {\Bbb F}_p^d$, we also conclude that $1 \leq k \leq p^d$. 

\begin{proposition} \label{proposition: equivalentFouriervanish}
Let $E$ be a subset of $\mathbb{F}_p^d$ and $m$ a fixed nonzero vector of $\mathbb{F}_p^d$. Then the following are equivalent: \begin{itemize} 

\item (1) $\hat{E}(m) = 0$. 

\item (2) $E$ is equidistributed on the hyperplanes $H_0, H_1, \dots, H_{p-1}$ where 
$$H_j = \{ x: x \cdot m = j \}.$$ 

\item (3) $\widehat{E}(rm) = 0$ for all $r \in \mathbb{F}_p^*$. 

\end{itemize} 

\end{proposition}
\begin{proof}
Suppose that (1)  holds. Then $0 = \hat{E}(m) = \frac{1}{p^d} \sum_{x \in E} \chi(-m \cdot x)$
where $\chi$ is a nontrivial additive character of $\mathbb{F}_p$. We may assume $\chi(-t) = \xi^t$ where $\xi$ is a primitive $p$th root of unity.

Then we have $0 = \sum_{k=0}^{p-1} a_k \xi^k$ where $a_k$ is the number of elements of $E$ on the hyperplane $H_k$.
This implies $\xi$ is a zero of the polynomial $a_0 + a_1x + a_2 x^2 + \dots + a_{p-1} x^{p-1}$ with rational coefficients $a_j$. However it is well known the minimal polynomial (over $\mathbb{Q}$) 
of $\xi$ is $1+x+ x+ \dots + x^{p-1} = \frac{x^p-1}{x-1}$ which is irreducible by Eisenstein's criterion. Thus 
$a_0 + a_1x + a_2x^2 + \dots a_{p-1}x^{p-1}$ is a $\mathbb{Q}$-multiple of $1+x+x^2 + \dots +x^{p-1}$ and so $a_0=a_1=a_2 = \dots = a_{p-1}$. In other words, $E$ is equidistributed over the hyperplanes $H_0, \dots, H_{p-1}$. Thus (2) holds. As all the steps in the argument are reversible (1) and (2) are equivalent.

Finally noting that replacing $m$ with $rm$, $r$ a nonzero scalar, does not change the corresponding set of hyperplanes, we see that (2) and (3) are equivalent.
\end{proof}

Note that this last proposition is not true for Fourier transforms of general complex valued functions over prime fields. Indeed one can create a complex function 
with any prescribed set of zeros on the Fourier space side and then use the inverse Fourier transform to get an example of a complex function 
where the equivalences of Proposition~\ref{proposition: equivalentFouriervanish} don't hold. However it is true for rational-valued
 functions such as characteristic functions 
-more results of this nature that hold only for rational-valued functions are discussed in Appendix~\ref{section: rational}.

\subsection{Proof of Theorem \ref{1tilecharacterization}} Note that $|E||A|=p^2$, so $|E|=1, p$ or $p^2$ if $|E|=1$ then $E$ is a point and we are done. If $|E|=p^2$ then $E$ is the whole plane and we are done, so without loss of generality $|E|=p$.

If $\hat{E}(m)$ never vanishes then by Lemma~\ref{lemma: novanish}, $E$ is a point and we are done. On the other hand if 
$\hat{E}(m)=0$ for some nonzero $m$, then it vanishes on $L$, the line passing through the origin and $m \not=\vec{0}$
.
Thus if we set $L^{\perp}$ to be the line through the origin, perpendicular to $m$, we see that 
$$\widehat{L^{\perp}}(s)\widehat{E}(s)=0$$ for all nonzero $s$. This is because by (\ref{subspaceft}) 
$\widehat{L^{\perp}}(s)=q^{-(d-1)} L(s)$. Since $|L^{\perp}|=p=|E|$ we then see that $E$ $1$-tiles ${\Bbb F}_p^2$ by $L^{\perp}$. 

Since $\widehat{E}(m)=0$ we see that $E$ is equidistributed on the set of $p$ lines $H_t=\{ x | x \cdot m = t\}, t \in \mathbb{F}_p$.  
Since $|E|=p$ this means there is exactly one point of $E$ on each of these lines. 

We will now choose a coordinate system in which $E$ will be represented as a graph of a function. The coordinate system will either be an orthogonal system 
or an isotropic system depending on the nature of the vector $m$.

Case 1: $m \cdot m \neq 0$: We may set $e_1 = m$ and $e_2$ a vector orthogonal to $m$. $\{ e_1, e_2 \}$ is an orthogonal basis because $e_2$ does not lie on the line throough $m$ as this line is not isotropic. If we take a general vector $\hat{x} = x_1 e_1 + x_2 e_2$ we see that $\hat{x} \cdot m = x_1 (m \cdot m)$ 
and so the lines $H_t, t \in \mathbb{F}_p$ are the same as the lines of constant $x_1$ coordinate with respect to this orthogonal basis $\{ e_1, e_2 \}$.
Thus there is a unique value of $x_2$ for any given value of $x_1$ so that $x_1 e_1 + x_2 e_2 \in E$. Thus 
$E=\{ x_1 e_1 + f(x_1) e_2: x_1 \in \mathbb{F}_p \}=Graph(f)$ 
for some function $f: \mathbb{F}_p \to \mathbb{F}_p$.

Case 2: $m \cdot m = 0$: We may set $e_1 = m$. In this case any vector orthogonal to $e_1$ lies on the line generated by $e_1$ and so cannot be part of a basis 
with $e_1$. Instead we select $e_2$ off the line generated by $e_1$ and scale it so that $e_1 \cdot e_2 = 1$. Then by subtracting a suitable multiple of $e_1$ from $e_2$ one can also ensure $e_2 \cdot e_2 = 0$. Thus $\{ e_1, e_2 \}$ is a basis 
consisting of two linearly independent isotropic vectors. With respect to this basis, the dot product is represented by the matrix 
$$
\mathbb{A} = \begin{bmatrix} 0 & 1 \\ 1 & 0 \end{bmatrix}
$$
which exhibits the plane as the hyperbolic plane. This case can only occur when $p=1$ mod $4$. 
Note when we express a general vector $\hat{x}=x_1 e_1 + x_2 e_2$ with respect to this basis we have $\hat{x} \cdot m = x_2$ thus the lines 
$\{H_t: t \in \mathbb{F}_p \}$ are the same as the lines of constant $x_2$ coordinate with respect to this basis and $E$ has a unique point on each of these lines.
Thus $E=\{ f(x_2) e_1 + x_2 e_2: x_2 \in \mathbb{F}_p \} = Graph(f)$ is a graph with respect to this isotropic coordinate system.

Finally we note any function $f: \mathbb{F}_p \to \mathbb{F}_p$ is given by a polynomial of degree at most $p-1$, explicitly expressed in the form 
$$ f(x) = \sum_{k \in \mathbb{F}_p} f(k) \frac{\Pi_{j \neq k} (x-j)}{\Pi_{j \neq k} (k-j)}$$
expresses $f$ as such a polynomial in $x$.

\vskip.125in 

\subsection{Proof of Theorem \ref{ktilecharacterization}} 

If $\widehat{E}(m)$ never vanishes then by Lemma~\ref{lemma: novanish} we are done so we may assume $\widehat{E}(m)=0$ for some nonzero $m$. Then $E$ vanishes on the nonzero elements of the line $L$ through $m$ and the origin. By proposition~\ref{proposition: equivalentFouriervanish}, $E$ is equidistributed on the $p$ lines parallel to $L^{\perp}$  with $s$ elements per line. Note $1 \leq s \leq p-1$ as if $s=p$ then $E$ is the whole plane already which was already covered. Then $|E|=ps$ and we can partition $E$ into $s$ sets $E_1, \dots, E_s$ each with exactly one element on each of the lines parallel to $L^{\perp}$.

As we argued above, this means $E_j=Graph(f_j)$ of some function $f_j: \mathbb{F}_p \to \mathbb{F}_p$ once we take coordinates $x$ parallel to $L$ and $y$ parallel to $L^{\perp}$. Since $|E||A|=kp^2$ we have $sp|A|=kp^2$ and so $|A|=\frac{kp}{s}$ so $s$ must divide $kp$. Since $s \leq p-1$, $s$ is relatively prime to $p$ and so $s$ divides $k$. This completes the proof of Theorem \ref{ktilecharacterization}. 

\vskip.25in 

\subsection{Proof of Theorem \ref{1tilecharacterization3space}}

We proceed as in the proof of Theorem~\ref{1tilecharacterization}. Let $(E,A)$ be a tiling pair in $\mathbb{F}_p^d$ 
As $|E||A|=p^d$ then $|E|=p^{d-1}$ implies $|A|=p$ so at least one of $|E|, |A|$ is $p$. WLOG we will assume $|E|=p$ as the argument when $|A|=p$ is similar with the roles of $E$ and $A$ swapped.

We know that $\hat{E}(m)\hat{A}(m)=0$ for all nonzero $m$. If $\hat{A}(m)=0$ for all nonzero $m$ then Fourier inversion shows that the characteristic function 
of $A$ is a constant and hence that $A$ is the whole space and $|E|=1$ which contradicts our assumptions. Thus $\hat{A}(m)$ is nonzero for some nonzero $m$ 
and hence $\hat{E}(m)=0$ for some nonzero $m$. Thus $E$ equidistributes on the $p$ hyperplanes perpendicular to $m$. Let $H=H_m$ be the hyperplane 
through the origin perpendicular to $m$ and let $L$ be a line such that $H \oplus L = \mathbb{F}_p^d$. Let $\{e_1,\dots, e_d\}$ be a basis 
such that $e_1$ lies in $L$ and $\{e_2, \dots, e_d \}$ lie in $H$. Now $E$ equdistributes on the $p$ hyperplanes of constant $e_1$-coordinate and so has 
exactly one element per hyperplane as $|E|=p$. Thus there is a function $f=(f_2,\dots,f_d): \mathbb{F}_p \to \mathbb{F}_p^{d-1}$ such that 
$E=\{ x_1 e_1 + f_2(x_1) e_2 + \dots + f_d(x_1) e_d | x_1 \in \mathbb{F}_p \}=Graph(f)$ and so the tiling $(E,A)$ is graphical as claimed.

Now in $3$ dimensions, $d=3$ and if $(E,A)$ is a tiling pair then $|E|=1,p,p^2,p^3$. The cases $|E|=1,p^3$ are trivial and trivially graphical while 
the cases $|E|=p,p^2$ fall under the last argument.

\vskip.25in

\section{Proof of Theorem \ref{icp}: Isotropic circle packings in the plane}

\vskip.125in 

Once we show that the $q$ circles of radius $c \neq 0$ centered at the $q$ points on an isotropic line $L \subseteq P$ are disjoint, we will have shown $q \leq P(q,c)$. Note that $q$ circles would account for $q(q-1)=q^2-q$ points as circles of nonzero radius have size $q-1$ when $q=1 \mod 4$.  Thus only $q$ unused points remain allowing for at most one more circle to be packed. Thus $P(q,c) \leq q+1$.

So it remains only to show the disjointness of two different $c$-circles centered on an isotropic line. Consider $(u,v)$ and $(s,t)$ two distinct elements on an isotropic line $L$. A point of intersection $(x,y)$ of the circles of radius $c \neq 0$ about $(u,v)$ and 
$(s,t)$ must satisfy the two equations
\begin{eqnarray*}
(x-u)^2 + (y-v)^2 &=& c \\
(x-s)^2 + (y-t)^2 &=& c
\end{eqnarray*}
As $L$ is isotropic we have $u^2+v^2=s^2+t^2=0$ so expanding the two equations and subtracting we get
$-2x(u-s)-2y(v-t) = 0$ i.e., $(x,y) \cdot (u-s, v-t) = 0$. As $(u-s, v-t)$ is a nonzero element of the isotropic line $L$ which is its own perp, we conclude that $(x,y) \in L$ and so $x^2+y^2=0$ and $(x,y) \cdot (u,v)=0$. Plugging these into the first equation of the two above, we find $0 = c$ which gives a contradiction. Thus the two circles of radius $c$ centered at different points of the isotropic line must be disjoint. This completes the proof of the theorem.

Note that any point on the isotropic line $L$ has distance zero from any other point on the line $L$ and hence does not lie on any circle of radius $c \neq 0$ centered on the line. Thus $L$ lies in the complement of the union of the circles in this packing. As $|L|=q=q^2-q(q-1)$ we conclude that $L$ is the complement of this packing.

Note the isotropic packing of the last theorem is maximal as one can show that the complementary isotropic line cannot contain any circle of nonzero radius. However this does not necessary preclude the existence of a completely different maximum circle packing achieving $P(q,c)=q+1$.

\vskip.125in 

\section{Proof of Theorem \ref{polysex}: Polynomial method results} 

\vskip.125in 
   
Translating the set $E$ by all elements of the set $A$ corresponds to the product of polynomials 
$$ \left(\sum_{e \in E} z^e \right) \left(\sum_{a \in A} z^a \right) \mod \cal I.$$ 

On the other hand, this corresponds to multi-tiling at level $k$ if and only if the latter product equals $k \sum_{f \in {\Bbb F}_q^d} z^f$, which is the required identity.

\vskip.125in 

\subsection{Proof of Corollary \ref{sumpolysex}}
Fixing an index $j$ and focusing on the variable $z_j$, we differentiate both sides of \eqref{polysex} with respect to $z_j$, and then multiply by $z_j$ (i.e. we apply the differential operator $z_j \frac{\partial}{\partial z_j}$), obtaining:
\begin{align}
 &    \left(\sum_{e \in E  } e_j z^e \right) \left(\sum_{a \in A} z^a \right)  
 +  \left(\sum_{e \in E  } z^e \right) \left(\sum_{a \in A} a_j z^a \right)  
  = k \frac{\partial}{\partial z_j} 
  \left( 
   \frac{z_1^q -1}{z_1-1} \frac{z_2^q -1}{z_2-1}   \cdots    \frac{z_d^q -1}{z_d-1}
   \right)        \label{differentialoperator}     \\  
    &=  k  \frac{z_1^q -1}{z_1-1} \frac{z_2^q -1}{z_2-1}   \cdots  \left( z_j + 2z_j^2 + 3z_j^3 + \cdots + (q-1)z_j^{q-1} \right) 
   \cdots   \frac{z_d^q -1}{z_d-1}, 
\end{align}

\noindent
We now let all $z_j =1$, obtaining
\begin{align}
 &    \left(\sum_{e \in E  } e_j  \right)|A| 
 +  \left(\sum_{a \in A} a_j  \right)  |E|
    =  k  q^{d-1} \frac{ (q-1)q}{2}        \mod q\\
    &=  0   \mod q.
\end{align}

\vskip.125in 

Putting together all $d$ of these identities (one for each index $j$) into vector form, we obtain the desired result.

\vskip.125in 

\subsection{Proof of Corollary   \ref{sumpolysex2}}

We may now continue to differentiate equation  \eqref{differentialoperator}, applying to it the operator 
$z_j \frac{\partial}{\partial z_j}$ once again.

\begin{align}
 &    \left(\sum_{e \in E  } e_j^2 z^e \right) \left(\sum_{a \in A} z^a \right)  
 + 2  \left(\sum_{e \in E  } e_j z^e \right) \left(\sum_{a \in A} a_j z^a \right) 
 +  \left(\sum_{e \in E  } z^e \right) \left(\sum_{a \in A} a_j^2 z^a \right)   \\  
    &=  k  \frac{z_1^q -1}{z_1-1} \frac{z_2^q -1}{z_2-1}   
    \cdots  \left( z_j + 2^2 z_j^2 + 3^2 z_j^3 + \cdots + (q-1)^2 z_j^{q-1} \right) 
   \cdots   \frac{z_d^q -1}{z_d-1}, 
\end{align}

\noindent
Again specializing to all $z_j =1$, we obtain
\begin{align}
 &    \left(\sum_{e \in E  } e_j^2  \right)|A| 
 +  \left(\sum_{a \in A} a_j^2  \right)  |E| 
 +  2 \left(\sum_{e \in E  } e_j \right)    \left(\sum_{a \in A} a_j  \right) \\
     &=  k  q^{d-1}   \left( 1 + 2^2  + 3^2  + \cdots + (q-1)^2  \right)     \mod q  \\
     &=  0   \mod q.
\end{align}

We note that we may continue to get arbitrarily many identities by differentiating arbitrarily many times with respect to each $z_j$.

\appendix

\section{Fourier Transform over prime fields of rational-valued functions}
\label{section: rational}

\subsection{Galois characterization}

In this appendix we record some further properties unique to Fourier transforms over prime fields of rational-valued functions. We first set up some notation.

Let $p$ be a prime, $\xi$ a primitive $p$th root of unity in $\mathbb{C}$. The cyclotomic extension $\mathbb{Q}(\xi)$ has degree $p-1$ over $\mathbb{Q}$ 
where an element of $\mathbb{Q}(\xi)$ can be written uniquely as 
$$a_0 + a_1 \xi + \dots + a_{p-2} \xi^{p-2} \ \text{with} \ a_j \in \mathbb{Q}.$$ Further powers of 
$\xi$ can always be reduced using the minimal polynomial relationship $1 + \xi + \dots + \xi^{p-1}=0$. 

It is well-known that the Galois group of this extension, $Gal(\mathbb{Q}(\xi)/\mathbb{Q})$ is isomorphic to the multiplicative group $\mathbb{F}_p^*$ where 
$g_r$, the Galois automorphism corresponding to $r \in \mathbb{F}_p^*$ is the field automorphism given by $g_r(1)=1, g_r(\xi)=\xi^r$. Thus 
for example 
$$g_r(\xi^j)=\xi^{rj} \ 1 \leq j \leq p-1.$$

Finally we identify the group algebra $\mathbb{Q}[\mathbb{F}_p^d]$ with the set of rational-valued functions on $(\mathbb{F}_p^d,+)$. It is well 
known that convolution of functions corresponds to group algebra multiplication though we will not use this but just use the notation for this set out of convenience.

Our next result completely characterizes the Fourier transform of rational-valued functions on $\mathbb{F}_p^d$:

\begin{theorem}[Characterization of Fourier transform of rational-valued functions]
\label{thm: rationalFourier}
 Fix $p$ a prime and let $\{ g_r \}_{r \in \mathbb{F}_p^*}$ denote 
$Gal(\mathbb{Q}(\xi)/\mathbb{Q})$.
Let $f: \mathbb{F}_p^d \to \mathbb{Q}$ be a rational-valued function. Let $m \in \mathbb{F}_p^d - \{ 0 \}$ be a nonzero vector. Then for all 
$r \in \mathbb{F}_p^*$ we have:
$$ \hat{f}(rm)=g_r(\hat{f}(m)). $$

In particular $\hat{f}(m) = 0$ implies $\hat{f}(rm)=0$ for all $r \in \mathbb{F}_p^*$.Furthermore if we choose a set $M \subseteq \mathbb{F}_p^d$ such that $M$ contains exactly one nonzero element from each line through the origin 
(so $|M| = \frac{p^d-1}{p-1}$) and set 
$$Ave(f) = \frac{1}{p^d} \sum f(x)  \in \mathbb{Q}=\hat{f}(0)$$ to be the average of $f$ then the map
$$ \Phi: \mathbb{Q}[\mathbb{F}_p^d] \to \mathbb{Q} \times \mathbb{Q}(\xi)^{|M|} $$ given by
$$\Phi(f)=(\hat{f}(0), (\hat{f}(m))_{m \in M})$$ is a $\mathbb{Q}$-vector space isomorphism.
\end{theorem}

\vskip.125in 

\begin{corollary} \label{cor: rationalTrace} Let $f \in \mathbb{Q}[\mathbb{F}_p^d]$ be a rational-valued function on $\mathbb{F}_p^d$. Then for any $m \in \mathbb{F}_p^d \backslash \{ 0\}$ we have 
$$ \sum_{r \in \mathbb{F}_p^*} \hat{f}(rm) = Tr(\hat{f}(m)) \in \mathbb{Q} $$ where $Tr: \mathbb{Q}(\xi) \to \mathbb{Q}$ is the Galois trace characterized as the $\mathbb{Q}$-linear map such that  $Tr(1)=p-1, Tr(\xi^j)=-1$ for $1 \leq j \leq p-1$. We also have 
$$ \sum_{r \in \mathbb{F}_p^*} |\hat{f}(rm)|^2 = Tr(|\hat{f}(m)|^2) \in \mathbb{Q}. $$
\end{corollary}

\vskip.125in 

\begin{corollary} \label{cor: average} Let $f \in \mathbb{Q}[\mathbb{F}_p^d]$ and let $\mu = Ave(f)=\hat{f}(0)$. For nonzero vector $m \in \mathbb{F}_p^d$ let 
$\mu_m(t)$ be the average of $f$ over the affine hyperplane $H_{m,t}=\{ x: x \cdot m =  t \}$ i.e., 
$$ \mu_m(t) = \frac{1}{p^{d-1}}\sum_{x \in H_{m,t} } f(x) \in \mathbb{Q}.$$ Then 
$$ {\Bbb E}[\mu_m] = \frac{1}{p} \sum_{t \in \mathbb{F}_p} \mu_m(t) = \mu \in \mathbb{Q}$$
and
$$ Var[\mu_m] = \frac{1}{p} \sum_{t \in \mathbb{F}_p} (\mu_m(t)-\mu)^2 = E[\mu_m^2] - E[\mu_m]^2 \in \mathbb{Q} $$
satisfy 
$$\hat{f}(m) = \frac{1}{p}\sum_{j=0}^{p-1} \mu_m(t) \xi^t \in \mathbb{Q}(\xi)$$ and 
$$ \frac{1}{p^d} \sum_{x \in \mathbb{F}_p^d} |f(x)|^2  = \mu^2 + \sum_{m \in M} Var[\mu_m] $$ where $M$ is a set that contains exactly one nonzero element from each line through the origin so $|M| = \frac{p^d-1}{p-1}$. Furthermore, 
$$Var[\mu_m] = Tr(|\hat{f}(m)|^2)=\sum_{r \in \mathbb{F}_p^*} |\hat{f}(rm)|^2$$ for all nonzero vectors $m$. \end{corollary}

\vskip.125in 

\subsection{Proof of Theorem \ref{thm: rationalFourier}} We have 
$\hat{f}(m) = \frac{1}{p^d} \sum_{x \in \mathbb{F}_p^d} f(x)\chi(-x \cdot m)$ where $\chi(-1)=\xi$. Therefore 
$$\hat{f}(m) = \frac{1}{p^d} \sum_x f(x)\xi^{x \cdot m} \in \mathbb{Q}(\xi).$$

Applying the Galois automorphism $g_r$ (which fixes $f(x)$ as $f(x) \in \mathbb{Q}$), we then get
$$ g_r(\hat{f}(m)) = \frac{1}{p^d} \sum_x f(x) (\xi^r)^{x \cdot m} = \frac{1}{p^d} \sum_x f(x) \xi^{x \cdot (rm)} = \hat{f}(rm) $$
as desired. Now the map $\Phi$ described in the theorem is clearly $\mathbb{Q}$-linear. It is injective as the collection 
$\hat{f}(0), \{ \hat{f}(m) \}_{m \in M}$ determines all the Fourier coefficients $\hat{f}(m)$ by Galois conjugation and hence determines $f$ by Fourier inversion.

Counting $\mathbb{Q}$-dimensions of domain and codomain of $\Phi$ we find that the $\mathbb{Q}$-dimension of the domain is $p^d$ whilst the $\mathbb{Q}$-dimension of codomain is 
$$1 + |M| dim_{\mathbb{Q}} \mathbb{Q}(\xi) = 1 + \frac{p^d-1}{p-1} (p-1) = p^d$$ and so their dimensions agree 
and $\Phi$ is a $\mathbb{Q}$-vector space isomorphism.

\vskip.125in 

\subsection{Proof of Corollary \ref{cor: rationalTrace}} 

By Theorem~\ref{thm: rationalFourier}, 
$$ \sum_{r \in \mathbb{F}_p^*} \hat{f}(rm) = \sum_{r \in \mathbb{F}_p^*} g_r (\hat{f}(m)) = Tr(\hat{f}(m)) $$
where the second equality follows from the definition of the Galois trace map and the fact that $\hat{f}(m) \in \mathbb{Q}(\xi)$.
Since $\{ \xi^j, 1 \leq j \leq p-1 \}$ is a complete set of galois conjugates over $\mathbb{Q}$, $Trace(\xi^j) = \xi + \dots + \xi^{p-1} = -1$ for any $1 \leq j \leq p-1$. As $1$ is in the base field, $Tr(1)=dim_\mathbb{Q}(\mathbb{Q}(\xi)) \cdot 1 = p-1$.

Finally note that when $p$ is odd, complex conjugation is in $Gal(\mathbb{Q}(\xi)/\mathbb{Q}) \cong \mathbb{F}_p^*$ and hence commutes with the $g_r$. On the other hand, when $p=2$ complex conjugation acts trivially on $\hat{f}(m) \in \mathbb{Q}$ and can be ignored. Thus we obtain 

\begin{eqnarray*}
\sum_{r \in \mathbb{F}_p^*} |\hat{f}(rm)|^2 &=& \sum_{r \in \mathbb{F}_p^*} \hat{f}(rm) \overline{\hat{f}(rm)} \\
&=& \sum_{r \in \mathbb{F}_p^*} g_r(\hat{f}(m)) \overline{g_r(\hat{f}(m))} \\
&=& \sum_{r \in \mathbb{F}_p^*} g_r(\hat{f}(m) \overline{\hat{f}(m)}) \\
&=& Tr(|\hat{f}(m)|^2)
\end{eqnarray*}

Note also that $|\hat{f}(m)|^2$ will live in the maximal real subfield of $\mathbb{Q}(\xi)$ over which $\mathbb{Q}(\xi)$ has degree two when $p$ is odd. 

\subsection{Proof of Corollary \ref{cor: average}} 

An easy computation establishes $\frac{1}{p} \sum_{t \in \mathbb{F}_p} \mu_m(t) = \mu$. Now note after setting $\chi(-1)=\xi$ as usual we get
$$ \hat{f}(m) = \frac{1}{p^d} \sum_{x} f(x) \xi^{m \cdot x} = \frac{1}{p^d} \sum_{t \in \mathbb{F}_p} \sum_{x \in H_{m,t}} f(x) \xi^t = \frac{1}{p} \sum_{t \in \mathbb{F}_p} \mu_m(t) \xi^t
$$
as claimed. By Plancherel's Theorem we get:

\begin{eqnarray*}
\frac{1}{p^d} \sum_{x \in \mathbb{F}_p^d} |f(x)|^2 &=& \sum_{m \in \mathbb{F}_p^d} |\hat{f}(m)|^2 \\
&=& |\hat{f}(0)|^2 + \sum_{m \in M} \sum_{r \in \mathbb{F}_p^*} |\hat{f}(rm)|^2 \\
&=& \mu^2 + \sum_{m \in M} Tr(|\hat{f}(m)|^2) \text{ by Corollary~\ref{cor: rationalTrace}}
\end{eqnarray*}

Now as we compute 

$$|\hat{f}(m)|^2 = (p^{-1} \sum_t \mu_m(t) \xi^t) (p^{-1} \sum_s \mu_m(s) \xi^{-s}) = p^{-2} \sum_{s,t} \mu_m(t)\mu_m(s) \xi^{t-s}$$
Applying $Tr$ to both sides and separating the terms where $t=s$ and where $t \neq s$ we get:

\begin{eqnarray*}
Tr(|\hat{f}(m)|^2) &=& p^{-2} \sum_t \mu_m(t)^2 (p-1) + p^{-2} \sum_{s \neq t} \mu_m(t)\mu_m(s) (-1) \\
&=& p^{-1} \sum_t \mu_m(t)^2 - p^{-2} \sum_{s,t} \mu_m(t)\mu_m(s) \\
&=& p^{-1} \sum_t \mu_m(t)^2 -  \mu^2 \\
&=& ( E[\mu_m^2] - E[\mu_m]^2) = Var[\mu_m] \\
\end{eqnarray*}

Plugging this into the last expression of Plancherel's Theorem completes the proof.

\enddocument